\documentclass{article}

\usepackage[english]{babel}
\usepackage{amsfonts}
\usepackage{tikz}
\usepackage{authblk}
\usetikzlibrary{positioning}
\usetikzlibrary{shapes.misc}

\tikzset{cross/.style={cross out, draw=black, minimum size=2*(#1-\pgflinewidth), inner sep=0pt, outer sep=0pt},
cross/.default={1pt}}

\usepackage[letterpaper,top=2cm,bottom=2cm,left=3cm,right=3cm,marginparwidth=1.75cm]{geometry}

\usepackage{amsmath,amscd}
\usepackage[utf8]{inputenc}
\usepackage{amssymb}
\usepackage{amsthm}
\usepackage{graphicx}
\usepackage{bbm}
\usepackage[colorlinks=true, allcolors=blue]{hyperref}
\newtheorem{theorem}{Theorem}[section]

\newtheorem{lemma}{Lemma}[section]
\newtheorem{remark}{Remark}[section]

\date{}

\newcommand{\cA}{{\mathcal A}}
\newcommand{\cC}{{\mathcal C}}
\newcommand{\cD}{{\mathcal D}}
\newcommand{\cT}{{\mathcal T}}
\newcommand{\cU}{{\mathcal U}}
\newcommand{\Z}{{\mathbb{Z}}}
\newcommand{\poly}{{\sf poly}}
\newcommand{\HKR}{{\sf HKR}}
\newcommand{\id}{{\sf id}}

\newcommand{\td}{{\sf td}}
\newcommand{\tI}{{\widetilde{I}}}
\newcommand{\tJ}{{\widetilde{J}}}
\newcommand{\tK}{{\widetilde{K}}}
\newcommand{\dbar}{{\overline{\partial}}}

\title{A remark on the B-model Comparison of Hodge structures}
\author{Junwu Tu~\footnote{Junwu Tu, Institute of Mathematical Sciences, ShanghaiTech University,
Shanghai, 201210, China.}}
\begin{document}
\maketitle

\begin{abstract}
In this note we record a comparison theorem on the B-model variation of semi-infinite Hodge structures. This result is considered a folklore theorem by experts in the field. We only take this opportunity to write it down. Our motivation is to apply it in the study of B-model categorical enumerative invariants. 
\end{abstract}

\paragraph{{\bf B-model comparison of Hodge structures.}} Let $X$ be a smooth and projective variety over $\mathbb{C}$. Then its cohomology group carries an integral polarized Hodge structures. Remarkably, a large part of this structure may be reconstructed only from the derived category of $X$. More precisely, let us denote by $\mathcal{C}_X$ a dg-enhancement of its derived category $D^b\big({\sf Coh}(X)\big)$. By~\cite{LunOrl}, such a dg-enhancement is essentially unique. Following Kontsevich-Soibelman~\cite{KonSoi} and Kazarkov-Kontsevich-Pantev~\cite{KKP}, one obtains a weaker version of Hodge structures
from $\mathcal{C}_X$ which was called nc-Hodge structures. The purpose of this note is to compare the nc-Hodge structures of $\cC_X$ with the classical Hodge structures of $X$.  By a series of works~\cite{Lod,Wei,Wei2,Kel}, we recall the following table of comparison results. 
\medskip
\begin{center}\label{table:1}
{\begin{tabular}{l|l}
\hline
Commutative & Non-commutative\\
\hline
Coherent cohomology $\oplus_{p-q=*} H^q(\Omega^p_X)$ & Hochschild homology $HH_*(\mathcal{C}_X)$ \\
\hline
Coherent cohomology $\oplus_{p+q=*} H^q(\Lambda^pT_X)$ & Hochschild cohomology $HH^*(\mathcal{C}_X)$ \\
\hline
De Rham cohomology $H^*(X,\Omega^\bullet_X),\;\; (*\in \Z/2\Z)$  & Periodic cyclic homology $HP_*(\mathcal{C}_X),\;\; (*\in \Z/2\Z)$ \\
\hline
Hodge filtration $F^\bullet H^*(X,\Omega^\bullet_X)$ & nc-Hodge filtration $F^\bullet HP_*(\mathcal{C}_X)$\\
\hline
\end{tabular}}
\end{center}
\medskip
Let us denote by $T_\poly:= \oplus_p \Lambda^pT_X[-p]$ and $A_\poly := \oplus_p \Omega^p_X[p]$, both endowed with the zero differential. The comparison map in the above table is given by the Hochschild-Konstant-Rosenberg~\cite{HKR} (HKR) isomorphisms:
\begin{align*}
    I^\HKR &: H^*(X,T_\poly) \rightarrow HH^*(\cC_X),\\
    I_{\HKR} &: HH_*(\cC_X) \rightarrow H^*(X,A_\poly).
\end{align*}
Motivated from Tsygan formality~\cite{Sho} in the algebraic case, C\u ald\u araru~\cite{Cal} conjectured the following comparison table.
\medskip
\begin{center}
\scalebox{0.9}{\begin{tabular}{l|l}
\hline
Pairing on $\oplus_{p-q=*} H^q(\Omega^p_X)$ &  Mukai pairing on $HH_*(\mathcal{C}_X)$ \\
\hline
$H^*(X,T_\poly)$ with exterior product & $HH^*(\mathcal{C}_X)$ with cup product\\
\hline
$H^*(X,A_\poly)$ as an $H^*(X,T_\poly)$-module & $HH_*(\cC_X)$ as an $HH^*(\cC_X)$-module \\
\hline
\end{tabular}}
\end{center}
\medskip
The comparison of pairing was proved by Ramadoss~\cite{Ram}. The comparison of product structures and module structures were later proved by Calaque-Rossi-Van den Bergh~\cite{CRV}, see also Calaque-Rossi~\cite{CalRos}. Remarkably the comparison map in C\u ald\u araru's conjecture  is no longer the HKR isomorphism, and the proof~\cite{CRV} requires using Kontsevich's formality map~\cite{Kon}. Explicitly, they are given by the following compositions:
\begin{align*}
    J^*&: H^*(X,T_\poly) \stackrel{\sqrt{\td(X)}}{\longrightarrow} H^*(X,T_\poly) \stackrel{I^\HKR}{\longrightarrow} HH^*(\cC_X),\\
    J_* &: HH_*(\cC_X) \stackrel{I_\HKR}{\longrightarrow} H^*(X,A_\poly) \stackrel{\sqrt{\td(X)}}{\longrightarrow} H^*(X,A_\poly).
\end{align*}
Here $\td(X)$ stands for the Todd class of $X$ defined by the series $\frac{z}{1-e^{-z}}$.  Following~\cite{CalRos,CRV}, let us also define a third pair of isomorphisms by
\begin{align*}
    K^*&: H^*(X,T_\poly) \stackrel{\sqrt{\td'(X)}}{\longrightarrow} H^*(X,T_\poly) \stackrel{I^\HKR}{\longrightarrow} HH^*(\cC_X),\\
    K_* &: HH_*(\cC_X) \stackrel{I_\HKR}{\longrightarrow} H^*(X,A_\poly) \stackrel{\sqrt{\td'(X)}}{\longrightarrow} H^*(X,A_\poly).
\end{align*}
In the above definition, $\td'(X)$ is the modified Todd class defined by the series $\frac{z}{e^{z/2}- e^{-z/2}}$. 

\medskip
The goal of this paper is to further extend the above comparison results. Since the chain-level Hochschild-Konstant-Rosenberg map is compatible intertwines the Connes operator with the de Rham differential, we also obtain an isomorphism denoted by
\[ \tI_{\HKR}: H^*\big(X,(C_\bullet(X)[[u]],b+uB)\big) \rightarrow H^*\big(X,(\Omega_X^\bullet[[u]],ud_{DR})\big).\]
Twisting by the square root of Todd class and the modified Todd class yields
\begin{align*}
\tJ &:H^*\big(X,(C_\bullet(X)[[u]],b+uB)\big) \stackrel{\tI_\HKR}{\longrightarrow} H^*\big(X,(\Omega_X^\bullet[[u]],ud_{DR})\big) \stackrel{\wedge \sqrt{{\sf td}(X)}}{\longrightarrow} H^*\big(X,(\Omega_X^\bullet[[u]],ud_{DR})\big),\\
\tK &:H^*\big(X,(C_\bullet(X)[[u]],b+uB)\big) \stackrel{\tI_\HKR}{\longrightarrow} H^*\big(X,(\Omega_X^\bullet[[u]],ud_{DR})\big) \stackrel{\wedge \sqrt{{\sf td}'(X)}}{\longrightarrow} H^*\big(X,(\Omega_X^\bullet[[u]],ud_{DR})\big).
\end{align*}
Furthermore, let us denote by $\int_X: H^n(X,\omega_X) \rightarrow \mathbb{C}$ the trace map from Lipman~\cite{Lip}. Note that by~\cite{SasTon}, we have
\begin{equation}\label{eq:signs}
\int_X = (-1)^{n(n-1)/2}(\frac{1}{2\pi i})^n\int_{X^{an}},
\end{equation}
with $\int_{X^{an}}: H^{2n}(X^{an},\mathbb{C}) \rightarrow \mathbb{C}$ the analytic integration map.
\begin{theorem}\label{thm:main}
    The following comparison results hold:
    \begin{enumerate}
        \item $\tJ$ intertwines higher residue pairings in the sense that
        \[ \langle \alpha, \beta \rangle_{\sf hres} = (-1)^{n(n+1)/2} \int_X \tJ(\alpha)\wedge \tJ(\beta^\vee),\]
        where $\vee$ is the canonical isomorphism denoted by $\vee:HC^-(\cC_X) \rightarrow HC^-(\cC_X^{\sf op})$. Under the HKR isomorphism, this map corresponds to $(-1)^p$ on the space of $p$-forms.
        \item Both $\tJ$ and $\tK$ intertwines the rational structures.
        \item In the family situation, the map $\tK$ intertwines the Gauss-Manin connection with the Getzler connection.
    \end{enumerate}
\end{theorem}

\begin{remark}
We make a few miscellaneous comments on conventions, signs and factors used in the literature that could be confusing from a first reading. Indeed,  Part (1.) is recently proved in~\cite{KK} which is based on~\cite{Kim}. Note that the formula obtained therein describes the so-called canonical pairing $\langle-,-\rangle_{{\sf can}}: HC^-(\cC_X) \otimes HC^-(\cC_X^{\sf op}) \rightarrow \mathbb{C}[[u]]$ (defined by Shklyarov~\cite{Shk}) by the following formula:
\[ \langle\alpha,\beta\rangle_{{\sf can}}= (-1)^{n(n+1)/2} \int_X I_\HKR(\alpha)\wedge I_\HKR(\beta)\wedge \td(X).\]
It implies $(1.)$ by precomposing with $\id\otimes \vee$. In Ramadoss~\cite{Ram}, a different trace map is used, which differs from our $\int_X$ by the sign $(-1)^{n(n+1)/2}$. Finally, throughout the literature the non-commuative Chern character (and hence the definition of Todd class $\td(X)$) differs from the usual Chern character by a factor $(-\frac{1}{2\pi i})^p$ at the $(p,p)$-component. This is compatible with Equation~\eqref{eq:signs} to deduce classical Hirzebruch-Riemann-Roch formula from (1.)~\cite{KK,Kim}. 
\end{remark}

\begin{remark}
Part (2.) is proved by Blanc~\cite{Bla}. Indeed, Blanc gave a purely categorical construction of topological K-theory of $X$, and showed that the HKR isomorphism intertwines the rational structures on $HP_*(\cC_X)$ and $H^*(X,\mathbb{C})$. We only need to observe that both $\sqrt{\td(X)}$ and $\sqrt{\td'(X)}$ are rational classes, hence $\tJ$ and $\tK$ still preserve the rational structure.
\end{remark}

\begin{remark}
    Note that the two Todd classes are related by $\td'(X)=\td(X) \exp (-c_1/2)$, in the Calabi-Yau case we obtain $\tJ=\tK$. 
    In particular, the theorem above verifies a conjecture of Ganatra-Perutz-Sheridan~\cite{GPS} that $\tJ$ is an isomorphism of polarized Variation of Semi-infinite Hodge Structures (VSHS). Furthermore, in the Calabi-Yau case, since $\td(X)$ is concentrated at $(2l,2l)$ components, Part (1.) of the theorem above implies the symmetry property of the higher residue pairing:
    \[\langle \alpha, \beta \rangle_{\sf hres} = (-1)^{n+ |\alpha||\beta|} \langle \beta,\alpha \rangle_{\sf hres}. \]
\end{remark}

\paragraph{{\bf Calaque-Rossi-Van den Bergh's proof of C\u ald\u araru's conjecture.}} By the previous remarks, it remains to prove part $(3.)$ of Theorem~\ref{thm:main}. Since the proof is along the same way as that of the C\u ald\u araru's conjecture in~\cite{CalRos,CRV}. We first briefly sketch the proof in {{\sl Loc. Cit.}}. The main idea is to use formal geometry~\cite{Fed,Dol} and fiberwise apply Kontsevich's and Shoikhet's formality map. Following~\cite{CalRos,CRV}, we introduce some notations:
\begin{itemize}
    \item Denote by $\cT_\poly$ the bundle of fiberwise holomorphic poly-vector fields on $X$.
    \item Denote by $\cA_\poly$ the bundle of fiberwise holomorphic differential forms on $X$.
    \item Denote by $D_\poly$ the bundle of holomorphic poly-differential operators on $X$, and $\cD_\poly$ the bundle of fiberwise holomorphic poly-differential operators on $X$.
    \item Denote by $\widehat{C}_\bullet$ the sheaf of continuous Hochschild chains on $X$, and $\widehat{\cC}_\bullet$ its fiberwise version.
\end{itemize}
It is well-known that these fiberwise bundles all admit a Fedosov~\cite{Fed} connection $D$ such that the associated de Rham complex of $D$ is a resolution of the corresponding non-fiberwise bundles. For example, in the case of poly-vector fields, there are quasi-isomorphisms
\[ \lambda:\big( \Omega^{0,*}(X,T_\poly), \dbar\big) \rightarrow   \big( \Omega(X,\cT_\poly), D \big), \;\; \pi:  \big( \Omega(X,\cT_\poly), D \big) \rightarrow \big( \Omega^{0,*}(X,T_\poly), \dbar\big).\]
The maps  $\lambda$, $\pi$ are the ``Taylor series" map and ``Projection to constant term" map respectively. One may verify both $\lambda$ and $\pi$ respects all existing algebraic structures on both sides, see~\cite[Section 10.2]{CalRos-book}.

Then, the desired comparison map $K^*: H^*(X,T_\poly) \rightarrow HH^*(X)$, in the chain level, is given by composing along the following diagram.

\begin{equation}\label{diag:dgla}
    \begin{CD}
    \big( \Omega(X,\cT_\poly), D \big) @> \cU_Q >> \big( \Omega(X,\cD_\poly), D  +d_H\big) \\
    @A \lambda AA        @VV \pi V \\
    \big( \Omega^{0,*}(X,T_\poly), \dbar\big) @> K^* >> \big( \Omega^{0,*}(X,D_\poly), \dbar+ d_H\big)
\end{CD}
\end{equation}

In the diagram above, we need to explain the construction of the map $\cU_Q$. Let us sketch this construction following~\cite{CalRos-book}. Let $U\subset X$ be a holomorphic coordinate chart. Then the Fedosov connection $D$ on both bundles $\cT_\poly$ is of the form
\[ D= \partial + \dbar + Q_U,\]
for some fiberwise vector fields valued one-form $Q_U\in \Omega^1(U,\cT)$. Then, on the open subset $U$, the map $\cU_Q$ is defined by 
\[ \cU_{Q_U} (\alpha)  := \sum_n \frac{1}{n!}\cU_{1+n}(\alpha,Q_U^{\otimes n}), \]
where $\cU: \cT_\poly \rightarrow \cD_\poly$ is the fiberwise Kontsevich formality map~\cite{Kon}. The upshot is that on a different chart $V$, the difference of the two fiberwise vector fields $Q_U-Q_V$ is a linear fiberwise vector field on $U\cap V$, which implies that we have
\begin{align*}
    \cU_{Q_U} (\alpha) & = \sum_n \frac{1}{n!}\cU_{1+n}(\alpha,Q_U^{\otimes n})\\
    & = \sum_n \frac{1}{n!}\cU_{1+n}(\alpha,Q_V^{\otimes n})\\
    & = \cU_{Q_V} (\alpha).
\end{align*}
In the second equality, we used the fact that Kontsevich's formality map vanishes whenever one of the input entry is a linear vector field. This calculation shows that we obtain a well-defined global map giving the desired chain map $\cU_Q: \big( \Omega(X,\cT_\poly), D \big) \rightarrow \big( \Omega(X,\cD_\poly), D \big)$. It is then a genius calculation~\cite{CalRos} to verify that indeed the induced map in cohomology of the composition $K^*:=\pi \cU_Q \lambda$ is indeed equal to $I^\HKR \circ \sqrt{\td'(X)}$.

Similar construction works in the comparison between $HH_*(X)$ and $H^*(X,A_\poly)$. In this case, we have the following diagram:
\begin{equation}\label{diag:hh}
\begin{CD}
    \big( \Omega(X,\widehat{\cC}_\bullet), D + b \big) @> \cU^{{\sf Sh}}_Q >> \big( \Omega(X,\cA_\poly), D \big) \\
    @A \lambda AA        @VV \pi V \\
    \big( \Omega^{0,*}(X,\widehat{C}_\bullet), \dbar+b\big) @> K_* >> \big( \Omega^{0,*}(X,A_\poly), \dbar\big)
\end{CD}
\end{equation}
Note that the comparison map is in the opposite direction $K_*: HH_*(X) \rightarrow H^*(X,A_\poly)$. Let us also briefly sketch the construction of the map $\cU^{{\sf Sh}}_Q$. Over a local holomorphic chart $U$, the element $Q_U$ acts on $\widehat{\cC}_\bullet)$ and $\cA_\poly$ by the Lie derivative action. Hence locally on $U$, we may use $Q_U$ to deform the Tsygan's formality map constructed by Shoihket~\cite{Sho}. This deformed map is given by
\[ \cU^{{\sf Sh}}_{Q_U} (-) := \sum_n \frac{1}{n!}\cU^{{\sf Sh}}_{1+n}(-,Q_U^{\otimes n}), \]
where $\cU^{{\sf Sh}}:\widehat{\cC}_\bullet \rightarrow \cA_\poly$ is the original fiberwise Tsygan's formality map~\cite{Sho}. Again, one can show these locally defined maps agree on overlaps, which yields a globally defined map $\cU_Q^{{\sf Sh}} : \big( \Omega(X,\widehat{\cC}_\bullet), D \big) \rightarrow \big( \Omega(X,\cA_\poly), D \big) $.

\paragraph{{\bf Finishing the proof.}} We need to set up the context to compare the Gauss-Manin connection with the Getzler connection. Indeed, let us consider a family of complex analytic deformations of $X$, i.e. a proper and submersive map $p: \mathfrak{X} \rightarrow \Delta$ from a complex manifold $\mathfrak{X}$ to a small open ball $\Delta\subset \mathbb{C}^\mu$ centered at the origin. We require that $p^{-1}(0)=X$. We think of $\mathfrak{X}$ as the total space of a deformation of $X$ parametrized by the open ball $\Delta$. 

By Ehresmann’s fibration theorem, one can always find a {\sl transversely holomorphic trialization}, i.e. a diffeomorphism
\[ (\phi, p): \mathfrak{X} \rightarrow X\times \Delta,\]
such that $\phi|_{p^{-1}(0)} = \id$, and fibers of $\phi: \mathfrak{X}\rightarrow X$ are holomorphic submanifolds transverse to $X$. 

Thus, we may push-forward the complex structure of $\mathfrak{X}$ to that of $X\times \Delta$ using the diffeomorphism $(\phi,p)$. Since in the transverse direction the fibers are always biholomorphic to $\Delta$, the complex structure is not changed in the $\Delta$ direction. After shrinking $\Delta$ if needed, the push-forward complex structure is equivalent to the following data
\[ \xi(t_1,\ldots,t_\mu) \in \Omega^{0,1}(X, T^{1,0}_X)\hat{\otimes} \mathcal{O}_\Delta,\]
with $t_1,\ldots,t_\mu$ holomorphic coordinates on $\Delta$ and that $\xi(0)=0$. Furthermore, the integrability of the complex structure is equivalent to the Maurer-Cartan equation
\[ \dbar \xi + \frac{1}{2}[ \xi, \xi] =0.\]
Now, back to the proof of $(3.)$ in Theorem~\ref{thm:main}, we may use the map $\lambda$ in Diagram~\eqref{diag:dgla} to push-forward $\xi(t)$ we obtain Maurer-Cartan elements $\lambda_*\xi$ which, as in the case of $Q_U$, yields also a fiberwise vector field valued one form. Deforming Diagram~\eqref{diag:hh} using the Maurer-Cartan elements $\xi(t)$ and $\lambda_*\xi(t)$ yields a new diagram:

\begin{equation}~\label{diag:deformation}
\begin{CD}
    \big( \Omega(X,\widehat{\cC}_\bullet)\hat{\otimes} \mathcal{O}_\Delta, D + b + L_{\lambda_*\xi} \big) @> \cU^{{\sf Sh}}_{Q,\lambda_*\xi} >> \big( \Omega(X,\cA_\poly)\hat{\otimes} \mathcal{O}_\Delta, D + L_{\lambda_*\xi} \big) \\
    @A \lambda AA        @VV \pi V \\
    \big( \Omega^{0,*}(X,\widehat{C}_\bullet)\hat{\otimes} \mathcal{O}_\Delta, \dbar+b+L_{\xi} \big) @> K_* >> \big( \Omega^{0,*}(X,A_\poly)\hat{\otimes} \mathcal{O}_\Delta, \dbar+L_\xi \big)
\end{CD}
\end{equation}

Here $\cU^{{\sf Sh}}_{Q,\lambda_*\xi}$ is the deformed Tsygan's formality map further deformed by the Maurer-Cartan element $\lambda_*\xi$. Explicitly, locally in a holomorphic coordinate chart on $U\subset X$, it is given by
\[ \cU^{{\sf Sh}}_{Q_U,\lambda_*\xi} (-) := \sum_{n,m} \frac{1}{n!m!}\cU^{{\sf Sh}}_{1+n+m}(-,Q_U^{\otimes n},(\lambda_*\xi)^{\otimes m}). \]
Since $\lambda_*\xi$ is globally defined, on another such chart $V$, we still have $ \cU^{{\sf Sh}}_{Q_U,\lambda_*\xi}= \cU^{{\sf Sh}}_{Q_V,\lambda_*\xi}$ because
$Q_U-Q_V$ is a linear fiberwise vector field on $U\cap V$. Thus we obtain a globally defined map $\cU^{{\sf Sh}}_{Q,\lambda_*\xi}$.

\medskip
Furthermore, it is proved by Willwatcher~\cite{Wil} that the Shoikhet formality map intertwines the Connes differential with the de Rham differential. Hence the above diagram extends to the following diagram:
\begin{equation} 
\begin{CD}
    \big( \Omega(X,\widehat{\cC}_\bullet)\hat{\otimes} \mathcal{O}_\Delta[[u]], D + b +uB+ L_{\lambda_*\xi} \big) @> \cU^{{\sf Sh}}_{Q,\lambda_*\xi} >> \big( \Omega(X,\cA_\poly)\hat{\otimes} \mathcal{O}_\Delta[[u]], D + ud_{DR}+L_{\lambda_*\xi} \big) \\
    @A \lambda AA        @VV \pi V \\
    \big( \Omega^{0,*}(X,\widehat{C}_\bullet)\hat{\otimes} \mathcal{O}_\Delta[[u]], \dbar+b+uB+L_{\xi} \big) @> \tK >> \big( \Omega^{0,*}(X,A_\poly)\hat{\otimes} \mathcal{O}_\Delta[[u]], \dbar+ud_{DR}+L_\xi \big)
\end{CD}
\end{equation}
Note that by definition $\tK:= \pi \circ  \cU^{{\sf Sh}}_{Q,\lambda_*\xi} \circ \lambda$.

There are naturally defined connection operators in the $t$-variables direction on the complexes in the above diagram, essentially all reflecting the fact that they are all trivialized in the $t$-direction. For example, let us consider the complex $\big( \Omega^{0,*}(X,A_\poly)\hat{\otimes} \mathcal{O}_\Delta[[u]], \dbar+ud_{DR}+L_\xi \big)$ in the lower right corner. Then we may define a connection operator
using the Maurer-Cartan element $\xi(t_1,\ldots,t_\mu)$ on it by 
\[ \nabla_{\frac{\partial}{\partial t_j}}^{{\sf GM}} := \frac{\partial}{\partial t_j} - \frac{\iota(\frac{\partial\xi}{\partial t_j})}{u}. \]
Similarly, on its fiberwise bundle we have the connection still denoted by $\nabla_{\frac{\partial}{\partial t_j}}^{{\sf GM}}:= \frac{\partial}{\partial t_j} - \frac{\iota(\frac{\partial\lambda_*\xi}{\partial t_j})}{u}$.
On the non-commutative side (i.e. the left column), following Getlzer~\cite{Get} we have 
\[ \nabla_{\frac{\partial}{\partial t_j}}^{{\sf Get}} := \frac{\partial}{\partial t_j} - \frac{\iota(\frac{\partial\xi}{\partial t_j})}{u},\]
with $\iota(\frac{\partial\xi}{\partial t_j})=b^{1|1}(\frac{\partial\xi}{\partial t_j}))+uB^{1|1}(\frac{\partial\xi}{\partial t_j}))$ by the non-commutative Calculus structure. The notations used here borrows that of Sheridan~\cite{She}. Again, we use the same notation for its fiberwise version $\nabla_{\frac{\partial}{\partial t_j}}^{{\sf Get}}:= \frac{\partial}{\partial t_j} - \frac{\iota(\frac{\partial\lambda_*\xi}{\partial t_j})}{u}$.

\begin{lemma}
The map $\tK$ intertwines the Getzler connection with the Gauss-Manin connection after taking cohomology.
\end{lemma}

\begin{proof}
Since both $\lambda$ and $\pi$ respects all existing algebraic structures on both sides, see~\cite[Section 10.2]{CalRos-book}, it follows that the maps $\lambda$ and $\pi$ intertwine the connection operators. Thus, we need to prove that 
\[ \nabla_{\frac{\partial}{\partial t}}^{{\sf GM}} \circ\cU^{{\sf Sh}}_{Q,\lambda_*\xi} - \cU^{{\sf Sh}}_{Q,\lambda_*\xi}\circ \nabla^{{\sf Get}}_{\frac{\partial}{\partial t}}=0,\]
after taking cohomology. 

On a local chart $U$, we can acheive this by a homotopy operator constructed by Cattaneo-Felder-Willwacher~\cite{CFW}. Indeed, consider the configuration space $E_{k+2,n+1}$ of $k+2$ interior marked points $(z_0,\cdots,z_{k+1})$, and $n+1$ boundary marked points $(z_{\overline{0}}, \cdots, z_{\overline{n}})$ on the unit disk, such that the first two interior points are such that $z_0=(0,0)$ and $z_1=(r,0)$ with $r\in [0,1]$. The notion of admissible graphs are defined in the same way as before. For each admissible graph, set
\[ c_\Gamma:= \int_{E_{k+2,n+1}} \bigwedge_{e\in E_\Gamma} d\theta_e.\]
Define an operator $H_{0,U}:\big( \Omega(X,\widehat{\cC}_\bullet)[[u]], d+  b +uB+ L_{Q_U +\lambda_*\xi} \big) \rightarrow  \big( \Omega(X,\cA_\poly)[[u]], d+ ud_{DR}+L_{Q_U + \lambda_*\xi} \big)$ by setting
\[ H_{0,U}(a_0|a_1|\cdots|a_n):= \sum_{k\geq 0, \Gamma \in E_{k+2,n+1}} \frac{1}{k!} c_\Gamma\cdot \cU^{\sf Sh}_\Gamma(\frac{d(\lambda_*\xi)}{dt},(\lambda_*\xi+Q_U)^{\otimes k}; a_0|a_1|\cdots|a_n). \]

The operator $\cU^{\sf Sh}_\Gamma(\frac{d(\lambda_*\xi)}{dt},(\lambda_*\xi+Q_U)^{\otimes k}; a_0|a_1|\cdots|a_n)$ (defined in the same way as in Shoikhet's formula) is illustrated in the following diagram:
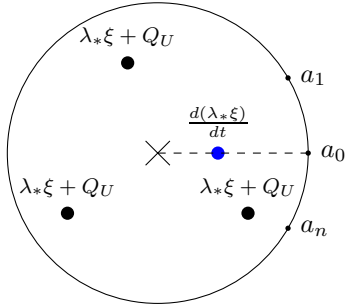
\[ \begin{tikzpicture}[baseline={(current bounding box.center)},scale=0.4]
\draw (0,0) circle (5);
\node[circle,fill=black,inner sep=0pt,minimum size=2pt,label=right:{$a_0$}] at (5,0) {};
\node[circle,fill=black,inner sep=0pt,minimum size=2pt,label=right:{$a_1$}] at (4.33,2.5) {};
\node[circle,fill=black,inner sep=0pt,minimum size=2pt,label=right:{$a_n$}] at (4.33,-2.5) {};
\draw (0,0) node[cross=5pt] {};
\node[circle,fill=blue,inner sep=0pt,minimum size=5pt,label=above:{\footnotesize $\frac{d(\lambda_*\xi)}{dt}$}] at (2,0) {};
\node[circle,fill=black,inner sep=0pt,minimum size=5pt,label=above:{\footnotesize $\lambda_*\xi+Q_U$}] at (-1,3) {};
\node[circle,fill=black,inner sep=0pt,minimum size=5pt,label=above:{\footnotesize $\lambda_*\xi+Q_U$}] at (-3,-2) {};
\node[circle,fill=black,inner sep=0pt,minimum size=5pt,label=above:{\footnotesize $\lambda_*\xi+Q_U$}] at (3,-2) {};
\draw[dashed] (0,0) to (5,0);
\end{tikzpicture}\]
Define another operator $H_{1,U}$ in a similar way with the first marked point (corresponding to the point where $\frac{d(\lambda_*\xi)}{dt}$ is inserted) constrainted between the framing of the central point $(0,0)$ and the boundary marked point $\overline{0}$. A typical configuration is illustrated in the following picture.
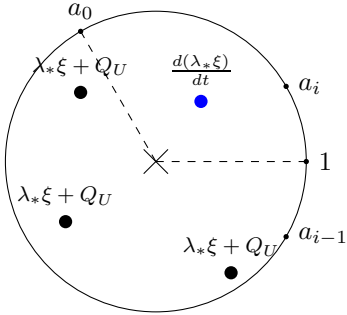
\[ \begin{tikzpicture}[baseline={(current bounding box.center)},scale=0.4]
\draw (0,0) circle (5);
\node[circle,fill=black,inner sep=0pt,minimum size=2pt,label=right:{ $1$}] at (5,0) {};
\node[circle,fill=black,inner sep=0pt,minimum size=2pt,label=right:{$a_i$}] at (4.33,2.5) {};
\node[circle,fill=black,inner sep=0pt,minimum size=2pt,label=right:{$a_{i-1}$}] at (4.33,-2.5) {};
\node[circle,fill=black,inner sep=0pt,minimum size=2pt,label=above:{$a_0$}] at (-2.5,4.33) {};
\draw (0,0) node[cross=5pt] {};
\draw[dashed] (0,0) to (-2.5,4.33);
\draw[dashed] (0,0) to (5,0);
\node[circle,fill=blue,inner sep=0pt,minimum size=5pt,label=above:{\footnotesize $\frac{d(\lambda_*\xi)}{dt}$}] at (1.5,2) {};
\node[circle,fill=black,inner sep=0pt,minimum size=5pt,label=above:{\footnotesize $\lambda_*\xi+Q_U$}] at (-2.5,2.3) {};
\node[circle,fill=black,inner sep=0pt,minimum size=5pt,label=above:{\footnotesize $\lambda_*\xi+Q_U$}] at (-3,-2) {};
\node[circle,fill=black,inner sep=0pt,minimum size=5pt,label=above:{\footnotesize $\lambda_*\xi+Q_U$}] at (2.5,-3.7) {};
\end{tikzpicture}\]
Then the following identity was proved in~\cite[Proposition 1.4 and Theorem 4.1]{CFW}:
\begin{align*} 
& u\big(\nabla_{\frac{\partial}{\partial t}}^{{\sf GM}}\pm \frac{1}{2} d_{DR}\iota(\frac{d(\lambda_*\xi)}{dt})d_{DR}\big)\circ \mathcal{U}^{\sf Sh}_{Q_U,\lambda_*\xi} - \mathcal{U}^{\sf Sh}_{Q_U,\lambda_*\xi}\circ (u\nabla^{{\sf Get}}_{\frac{\partial}{\partial t}})\\
& =  H_U \circ (d+b+uB+L_{Q_U+\lambda_*\xi}) + (L_{Q_U+\lambda_*\xi}+d+ud_{DR})\circ H_U,
\end{align*}
where $H_U=H_{0,U}+uH_{1,U}$. 

Furthermore, since $\xi$ is global, and $Q_U-Q_V$ is a linear fiberwise vector field on $U\cap V$, we have $H_U= H_V$ on the intersection $U\cap V$. Hence the locally defined Homotopy operators glue to form a globally well-defined map $H$ such that 
\begin{align*} 
& u\big(\nabla_{\frac{\partial}{\partial t}}^{{\sf GM}}\pm \frac{1}{2} d_{DR}\iota(\frac{d(\lambda_*\xi)}{dt})d_{DR}\big)\circ \mathcal{U}^{\sf Sh}_{Q,\lambda_*\xi} - \mathcal{U}^{\sf Sh}_{Q,\lambda_*\xi}\circ (u\nabla^{{\sf Get}}_{\frac{\partial}{\partial t}})\\
& =  H \circ (D+b+uB+L_{\lambda_*\xi}) + (D+L_{\lambda_*\xi}+ud_{DR})\circ H.
\end{align*}
This implies, after taking cohomology, we obtain the identity 
\[ u\big(\nabla_{\frac{\partial}{\partial t}}^{{\sf GM}}\pm \frac{1}{2} d_{DR}\iota(\frac{d(\lambda_*\xi)}{dt})d_{DR}\big)\circ \mathcal{U}^{\sf Sh}_{Q,\lambda_*\xi} = \mathcal{U}^{\sf Sh}_{Q,\lambda_*\xi}\circ (u\nabla^{{\sf Get}}_{\frac{\partial}{\partial t}}). \]
But since $X$ is smooth and proper, it satisfies Hodge-to-de-Rham degeneration, hence in cohomology the operator $d_{DR}$ acts by zero, which then implies our desired identity
\[ (u\nabla_{\frac{\partial}{\partial t}}^{{\sf GM}})\circ \mathcal{U}^{\sf Sh}_{Q,\lambda_*\xi} = \mathcal{U}^{\sf Sh}_{Q,\lambda_*\xi}\circ (u\nabla^{{\sf Get}}_{\frac{\partial}{\partial t}}).\]
\end{proof}

Finally, we observe that the map $\tK$, when restricted to $t\in \Delta$, is indeed given by $\sqrt{\td'(X_t)}\circ \tI_\HKR$. This is because in Diagram~\ref{diag:deformation} the complex structure deformation $\xi(t)$ may be absorbed into the $\dbar$ operator so that $\dbar_\xi:=\dbar+L_{\xi(t)}$ is the Dolbeault operator on the deformed space $X_t$. Thus, we have an isomorphic diagram
\begin{equation}~\label{diag:deformation2}
\begin{CD}
    \big( \Omega(X_t,\widehat{\cC}_\bullet)\hat{\otimes} \mathcal{O}_\Delta, D_t + b \big) @> \cU^{{\sf Sh}}_{Q_t} >> \big( \Omega(X_t,(\cA_t)_\poly)\hat{\otimes} \mathcal{O}_\Delta, D_t \big) \\
    @A \lambda AA        @VV \pi V \\
    \big( \Omega^{0,*}(X_t,\widehat{C}_\bullet)\hat{\otimes} \mathcal{O}_\Delta, \dbar_\xi+b \big) @> K_* >> \big( \Omega^{0,*}(X_t,(A_t)_\poly)\hat{\otimes} \mathcal{O}_\Delta, \dbar_\xi  \big)
\end{CD}
\end{equation}
where locally on $U\subset X_t$, we have a Fedesov operator of the form $D_t = \partial_t + \dbar_{\xi(t)} + Q_{U,t}$. This shows that we are in the same setup as that of Diagram~\ref{diag:hh}, in which case the computation of~\cite{CRV,CalRos} shows that in cohomology we have $\tK=\sqrt{\td'(X_t)}\circ \tI_\HKR$. This finishes the proof of part $(3.)$ of Theorem~\ref{thm:main}.

\medskip
\noindent {{\bf Acknowledgements.}} I would like to thank Sasha Polishchuk for bringing Bumsig Kim's work~\cite{Kim} to my attention. I am also very grateful to Taejung Kim to explain the signs used in the works~\cite{Kim, KK}, as well as sharing the reference~\cite{SasTon}. This work is partially supported by the NSFC grant 12071290.

\end{document}